\documentclass[12pt]{amsart}
\usepackage{amsopn}
\usepackage{amssymb}
\usepackage{graphicx}
\usepackage[all]{xy}
\usepackage{amscd}
\usepackage{color}
\usepackage[colorlinks]{hyperref}
 \newtheorem{thm}{Theorem}[section]

 \newtheorem{prop}[thm]{Proposition}
 \newtheorem{rem}[thm]{Remark}
 \newtheorem{defn}[thm]{Definition}

 \theoremstyle{definition}
 \theoremstyle{remark}
\allowdisplaybreaks

\begin{document}

\title[Operator representation of frames]
{Study on operator representation of frames in Hilbert spaces}

\author[J. Cheshmavar]{Jahangir Cheshmavar$^{*}$}

\address{Department of Mathematics, Payame Noor University (PNU), P.O.Box 19395-3697, Tehran, Iran}

\email{j$_{_-}$cheshmavar@pnu.ac.ir}

\author[A. Dallaki]{ Ayyaneh Dallaki}

\address{Department of Mathematics, Payame Noor University (PNU), P.O.Box 19395-3697, Tehran, Iran}

\email{ayyanehdallaki@student.pnu.ac.ir}


\keywords{Frames, Operator representation of frames, Spectrum,
Spectral radius.\\
\indent $^{*}$ Corresponding author} \maketitle

\begin{abstract}
The purpose of this paper is to give an overview of the operator
structure of frames $\{f_k\}_{k=1}^{\infty}$ as
$\{T^k\phi\}_{k=0}^{\infty}$, where the operator $T:\mathcal{H}
\rightarrow \mathcal{H}$ belongs to certain classes of linear
operators and the element $\phi$ belongs to $\mathcal{H}$. We
discuss the size of the set of such elements. Also, for a given
frame $\{f_k\}_{k=1}^{\infty}$ and any $n\in \mathbb{N}$, some
results are obtained for
$T^n(\{f_k\}_{k=1}^{\infty})=\{T^nf_1,T^nf_2,\cdots\}$. Finally,
we conclude this note by raising several questions connecting
frame theory and operator theory.
\end{abstract}

\section{\textbf{Introduction }}
The system of iterations $\{T^k\phi\}_{k=0}^{\infty}$, where $T$
is a bounded linear operator on a separable Hilbert space
$\mathcal{H}$ and $\phi \in \mathcal{H}$, the so-called dynamical
sampling problem is a relatively new research topic in Harmonic
analysis. This topic has been studied since the work of Aldroubi
and Petrosyan \cite{Aldroubi1.2016} for some new results
concerning frames and Bessel systems, see for example Aldroubi et
al. \cite{Aldroubi2.2017} to study dynamical sampling problem in
finite dimensional spaces. Christensen et al.
\cite{Christensen.2017,Christensen.2019,Christensen.2019.} studied
some properties of systems arising via iterated actions of
operators and also, frame properties of operator orbits. For more
details, we refer the interested reader to the
\cite{Aldroubi1.2016,Bayart,Christensen.2016}.


Let $F:=\{f_k\}_{k=1}^{\infty}$ be a frame for $\mathcal{H}$ which
spans an infinite dimensional subspace. A natural question to ask
is whether there exists a linear operator $T$ such that
$f_{k+1}=Tf_k$, for all $k\in \mathbb{N}$?

In \cite{Christensen.2019.}, it was proved that such an operator
exists if and only if $F$ is linearly independent; also, $T$ is
bounded if and only if the kernel of the synthesis operator of $F$
is invariant under the right shift operator on
$\ell^2(\mathbb{N})$, in the affirmative case,
$F=\{T^kf_1\}_{k=0}^{\infty}$.

In this note, for given frame $F$, some results are obtained for
$T^n(F)=\{T^nf_1,T^nf_2,\cdots\}$. Note that the problem
considered in this note is of some interest from other points of
view. Indeed, assuming that the system
$\{T^k\phi\}_{k=0}^{\infty}$ is a frame for $\mathcal{H}$. It is
natural to ask for a characterization of $\mathcal{V}(T)$: the set
of all $\phi\in \mathcal{H}$ such that
$\{T^k\phi\}_{k=0}^{\infty}$ is a frame for $\mathcal{H}$. In
\cite{Christensen.2019}, $\mathcal{V}(T)$ is obtained by applying
all invertible operators from the set of commutant ${T}^{'}$ of
$T$ to $\phi$. The chief aim of this paper is to understand the
size of the set $\mathcal{V}(T)$. Also, we show that the set of
all invertible operators $T \in B(\mathcal{H})$ for which
$\{T^k\phi\}_{k=0}^{\infty}$ is a frame for $\mathcal{H}$ for some
$\phi \in \mathcal{H}$, is open. Finally, we close this work by
raising some questions connecting frame theory and operator
theory.


In the following we denote by $\mathcal{H}$ a separable Hilbert
space. We use $B(\mathcal{H})$ for the set of all bounded linear
operators on $\mathcal{H}$, $\mathbb{N}$ for the natural numbers
as the index set and we pick $\mathbb{N}_0:=\mathbb{N} \cup
\{0\}$. The spectrum of an operator $T \in B(\mathcal{H})$ is
denoted by $\sigma(T)$, which is defined as $$\sigma(T)=\{\lambda
\in \mathbb{C}: T-\lambda I \,\ \mbox{is not invertible}\},$$ and
the spectral radius of $T$ is denoted by $r(T)$, which is defined
as
$$r(T)=\sup\{|\lambda|: \lambda \in \sigma(T)\}.$$ Given an
operator $\Lambda$, we denote its domain by $dom(\Lambda)$ and its
range by $ran(\Lambda)$.
\smallskip
\begin{defn}\cite{Christensen.2016}
 A sequence of vectors $F$ in
$\mathcal{H}$ is a \textit{frame} for $\mathcal{H}$ if there exist
constants $A,B>0$ so that
\begin{eqnarray}
\label{100} A \parallel f\parallel^2\leq
\sum_{k=1}^{\infty}\mid\langle f, \;f_k\rangle \mid^2\leq B
\parallel f\parallel^2~,
\end{eqnarray}
for all $f \in \mathcal{H}$.
\end{defn}
It follows from the definition that if $F$ is a frame for
$\mathcal{H}$, then
\begin{eqnarray}\label{2}
\overline{span}F=\mathcal{H}~.
\end{eqnarray}

If $A=B$, then $F$ is called a \textit{tight frame} and if $A =B
=1$, then $F$ is called a \textit{Parseval frame or a normalized
tight frame}. Also, $F$ is called a \textit{Bessel sequence} if
at least the upper condition in (\ref{100}) holds.\\

For any sequence $F$ in $\mathcal{H}$, the associated
\textit{synthesis operator} is defined by
\begin{eqnarray*}
U_F:D_1(F) \rightarrow \mathcal{H};\quad\
U_F(\{c_k\}_{k=1}^{\infty})=\sum_{k=1}^{\infty}c_k f_k~,
\end{eqnarray*}
 where
\begin{eqnarray*}
D_1(F):=\left\{\{c_k\}_{k=1}^{\infty}\in \ell^2(\mathbb{N}):\,\
\sum_{k=1}^{\infty}c_kf_k \,\ \mbox{converges in}\,\
\mathcal{H}\right\}.
\end{eqnarray*}
The \textit{analysis operator} is defined by
\begin{eqnarray*}
U^{\ast}_F:D_2(F) \rightarrow \ell^2(\mathbb{N});\quad\ U^{\ast}_F
f=\{\langle f, \;f_k \rangle\}_{k=1}^{\infty}~,
\end{eqnarray*}
 where
\begin{eqnarray*}
D_2(F):=\left\{f \in \mathcal{H}: \,\ \{\langle f,
\;f_k\rangle\}_{k=1}^{\infty}\in \ell^2(\mathbb{N})\right\}.
\end{eqnarray*}

Also, the \textit{frame operator} is defined by
\begin{eqnarray*}
S_F:D_3(F) \rightarrow \mathcal{H};\quad\
S_Ff=\sum_{k=1}^{\infty}\langle f, \;f_k \rangle f_k~,
\end{eqnarray*}
 where
\begin{eqnarray*}
D_3(F):=\left\{f \in \mathcal{H}: \,\ \sum_{k=1}^{\infty}\langle
f, \;f_k\rangle f_k \,\ \mbox{converges in}\,\
\mathcal{H}\right\}.
\end{eqnarray*}
\smallskip
\begin{rem}
It is known that if $F$ is at least a Bessel sequence, then
\begin{itemize}
\item [{(i)}] the operator $U_F$ is well-defined and
$D_1(F)=\ell^2(\mathbb{N})$, \item [{(ii)}] the operators $S_F$
and $U^{\ast}_F$ are well-defined and the domain of both is
$\mathcal{H}$.
\end{itemize}
\end{rem}

\begin{defn}\cite{Christensen.2016}
A sequence of vectors $F$ in $\mathcal{H}$ is a \textit{Riesz
basis} for $\mathcal{H}$ with bounds $0<A\leq B<\infty$, if $F$
satisfy (\ref{2}) and for every scalar sequence $(c_k)_k \in
\ell^2$, one has
\begin{eqnarray}
\label{10} A \sum_k |c_k|^2\leq \|\sum_k c_kf_k\|^2\leq B \sum_k
|c_k|^2~.
\end{eqnarray}
\end{defn}

We have the following result which is well known in frame theory
in general case, (see for example \cite{Blan.2001}), we include a
sketch of the proof.

\begin{prop}
Suppose that $F$ is a frame (not tight frame) for a Hilbert space
$\mathcal{H}$ with bounds $A$ and $B$, respectively, and frame
operator $S_F$. If $f_{j} \in F$ such that
$||S_F^{-1/2}f_{j}||\leq \sqrt{\frac{A}{B}}$, then $F\setminus
\{f_{j}\}$ is a frame for $\mathcal{H}$.
\end{prop}
\begin{proof}
Let $f_{j} \in F$ with $||S_F^{-1/2}f_{j}||\leq
\sqrt{\frac{A}{B}}$. It is known that
$\{S_F^{-1/2}f_{k}\}_{k=1}^{\infty}$ is a Parseval frame for
$\mathcal{H}$, (see \cite[Corollary 6.3.5]{Heil.1990} and
\cite[Theorem III.2]{Casazza.2000}), so for any $h\in
\mathcal{H}$, we obtain
\begin{eqnarray*}
\sum_{k=1}^{\infty}|\langle h, \;S_F^{-1/2}f_{k}
\rangle|^{2}-|\langle h, \;S_F^{-1/2}f_{j} \rangle|^{2}&\geq&
||h||^{2}-||S_F^{-1/2}f_{j}||^{2}||h||^{2} \\
&\geq& (1-\frac{A}{B})||h||^{2}.
\end{eqnarray*}
That is, the sequence $S_F^{-1/2}(F\setminus \{f_{j}\})$ is a
frame for $\mathcal{H}$ with lower frame bound $1-\frac{A}{B}$.
Since $S_F^{-1/2}$ is an invertible operator on $\mathcal{H}$,
thus $F\setminus \{f_{j}\}$ is  a frame for $\mathcal{H}$ as well.
\end{proof}

\bigskip

The availability of the representation
$F=\{T^kf_1\}_{k=0}^{\infty}$ is characterized in
\cite{Christensen.2019.}:

\begin{prop}
Consider any sequence $F$ in $\mathcal{H}$ for which $span F$ is
infinite-dimensional. Then the following are equivalent:
\begin{itemize}
\item[(i)] $F$ is linearly independent. \item[(ii)] There exists a
linear operator $T:span F \rightarrow \mathcal{H}$ such that

$F=\{T^kf_1\}_{k=0}^{\infty}$~.
\end{itemize}
\end{prop}

\section{\textbf{Main results}}\label{Sec4}
We begin with the following notations. Note that the sequence $F$
being represented by $T$ means that
\begin{eqnarray*}
F:=\{f_k\}_{k=1}^{\infty}=\{f_1,f_2,f_3 \cdots\}=\{f_1,Tf_1,T^2f_1
\cdots\}=\{T^kf_1\}_{k=0}^{\infty}.
\end{eqnarray*}
Therefore, for any $n\in \mathbb{N}$, we have
\begin{eqnarray*}
T_n(F):=T^n(F)=\{T^nf_k\}_{k=1}^{\infty}&=&\{T^nf_1,T^nf_2,\cdots\}\\
&=&\{f_{n+1},f_{n+2},\cdots\}=F\backslash \{f_1,f_2,\cdots, f_n\}.
\end{eqnarray*}

\bigskip

Throughout this section, for any $n\in \mathbb{N}$ we simply pick
$T_n:=T^n$. Suppose that $F$ is a sequence in $\mathcal{H}$ of the
form $\{T^kf_1\}_{k=0}^{\infty}$, for some operator $T$. For any
$n\in \mathbb{N}$, the associated synthesis, analysis, and frame
operators of $T_n(F)$ are given by the following proposition:
\begin{prop}\label{lemma.1}
The synthesis, the analysis and the frame operators for $T_n(F)$
are given by $T_nU_F, \,\ U^{\ast}_FT_n^{\ast}$ and
$T_nS_FT_n^{\ast}$, respectively.
\end{prop}
\begin{proof}
For each $n\in \mathbb{N}$, we have

\begin{eqnarray*}
U_{T_n(F)}(\{c_k\}_{k=1}^{\infty})=\sum_{k=1}^{\infty}c_kT_nf_k=\sum_{k=1}^{\infty}T_n(c_kf_k)=T_nU_F(\{c_k\}_{k=1}^{\infty}),
\end{eqnarray*}
it follows that $U_{T_n(F)}=T_nU_F$.\\
For the analysis operator, we have

\begin{eqnarray*}
U^{\ast}_{T_n(F)}f=\{\langle f, \;T_nf_k
\rangle\}_{k=1}^{\infty}=\{\langle T_n^{\ast}f, \;f_k
\rangle\}_{k=1}^{\infty}=U^{\ast}_FT_n^{\ast}f,
\end{eqnarray*}
it follows that $U^{\ast}_{T_n(F)}=U^{\ast}_FT_n^{\ast}$.\\
Also, we have
\begin{eqnarray*}
S_{T_n(F)}f=\sum_{k=1}^{\infty}\langle f, \;T_nf_k \rangle
T_nf_k=T_n\sum_{k=1}^{\infty}\langle T_n^{\ast}f, \;f_k\rangle
f_k=T_nS_FT_n^{\ast}f,
\end{eqnarray*}
it follows that $S_{T_n(F)}=T_nS_FT_n^{\ast}$.
\end{proof}

\bigskip

It is well-known that $F$ is a Riesz basis for $\mathcal{H}$ if
and only if the analysis operator $U^{\ast}_F$ is bijective
(because by \cite[Theorem (5.4.1)]{Christensen.2016} and
\cite[Proposition (5.1.5)]{Grochening.2001}, $F$ is a Riesz basis
if and only if $F$ is a frame and $U^{\ast}_F$ is surjective and
also, by \cite[Corollary (5.5.3)]{Christensen.2016} $F$ is a frame
if and only if $U^{\ast}_F$ is an injective operator with closed
range).

\smallskip

We now have the following result:

\begin{prop}
Let $F$ be a Riesz basis for $\mathcal{H}$ of the form
$\{T^kf_1\}_{k=0}^{\infty}$, for some operator $T$. Then for any
$n \in \mathbb{N}$, $T_n(F)$ is a Riesz basis for $\mathcal{H}$ if
and only if $T_n$ is an invertible operator on $\mathcal{H}$.
\end{prop}
\begin{proof}
Since $F$ is a Riesz basis for $\mathcal{H}$, then the analysis
operator $U^{\ast}_F$ is invertible. For any $n \in \mathbb{N}$,
if $T_n$ is an invertible operator on $\mathcal{H}$ then
$U^{\ast}_FT^{\ast}_n$ (the analysis operator of $T_n(F)$) is
invertible. Therefore, for any $n \in \mathbb{N}$, $T_n(F)$ is a
Riesz basis for $\mathcal{H}$. Conversely, for any $n \in
\mathbb{N}$ if $T_n(F)$ is a Riesz basis for $\mathcal{H}$, then
the analysis operator $U^{\ast}_FT^{\ast}_n$ is invertible. Also,
since the analysis operator $U^{\ast}_F$ is an invertible operator
on $\mathcal{H}$, it follows that $T^{\ast}_n$ and then $T_n$ is
an invertible operator on $\mathcal{H}$.
\end{proof}

\bigskip

We give some relations associated to the domain and range of the
synthesis, analysis and frame operators for $T_n(F)$.
\begin{prop}
For any $n\in \mathbb{N}$, the following statements hold:
\begin{itemize}
\item [(i)] $dom (T_nS_FT_n^{\ast})=dom (S_{T_n(F)})\subseteq
dom(U^{\ast}_{T_n(F)})=dom(U_F^{\ast}T_n^{\ast})$, \item
[(ii)]$ran(T_nS_FT_n^{\ast})=ran (S_{T_n(F)})\subseteq
ran(U_{T_n(F)})=ran(T_nU_F)$, \item[(iii)] $dom (S_{T_n(F)})= dom
(U^{\ast}_{T_n(F)})$ \mbox{iff} $ran (U^{\ast}_{T_n(F)}) \subseteq
dom (U_{T_n(F)})$.
\end{itemize}
\end{prop}
\begin{proof}
For any $n\in \mathbb{N}$, it is known that
\begin{eqnarray*}
dom(T_nS_FT_n^{\ast})=dom(S_{T_n(F)})=\{f \in \mathcal{H}: \,\
\sum_{k=1}^{\infty}\langle f, \;T_nf_k\rangle T_nf_k \,\
\mbox{converges in}\,\ \mathcal{H} \},
\end{eqnarray*}
\begin{eqnarray*}
dom(U_F^{\ast}T_n^{\ast})=dom(U^{\ast}_{T_n(F)})=\{f \in
\mathcal{H}: \,\ \{\langle f, \;T_nf_k\rangle\}_{k=1}^{\infty}\in
\ell^2(\mathbb{N}) \},
\end{eqnarray*}
\begin{eqnarray*}
 dom(T_nU_F)&=&dom(U_{T_n(F)})\\
&=&\{ (c_k)_{k=1}^{\infty}\in \ell^2(\mathbb{N}): \,\
\sum_{k=1}^{\infty}c_kT_nf_k=\sum_{k=1}^{\infty}T_n(c_kf_k) \,\
\mbox{converges in}\,\ \mathcal{H} \}.
\end{eqnarray*}

For (i), let $f \in dom (S_{T_n(F)})$. It follows that
\begin{eqnarray*}
\sum_{k=1}^{\infty}\langle f, \;T_nf_k\rangle T_nf_k \,\
\mbox{converges in}\,\ \mathcal{H}.
\end{eqnarray*}
Then we obtain that
\begin{eqnarray*}
\left\langle \sum_{k=1}^{N}\langle f, \;T_nf_k \rangle T_nf_k, \;f
\right\rangle \longrightarrow \langle S_{T_n(F)}f, \;f \rangle,\,\
\mbox{as}\,\ N \longrightarrow \infty,
\end{eqnarray*}
which implies that
\begin{eqnarray*}
 \sum_{k=1}^{N}|\langle f, \;T_nf_k \rangle |^2=\sum_{k=1}^{N}|\langle T_n^{\ast}f, \;f_k \rangle |^2 \,\ \mbox{converges
 as}\,\ N\longrightarrow \infty.
\end{eqnarray*}
Therefore, $f \in dom (U^{\ast}_{T_n(F)})=dom
(U_F^{\ast}T_n^{\ast})$.\\

(ii) follows from the fact that
$$S_{T_n(F)}=T_nS_FT_n^{\ast}=T_nU_FU_F^{\ast}T_n^{\ast}=\underbrace{T_nU_F}U^{\ast}_{T_n(F)}.$$
\smallskip

For (iii), let $dom (S_{T_n(F)})= dom (U^{\ast}_{T_n(F)})$ and
take $U^{\ast}_{T_n(F)}f \in ran (U^{\ast}_{T_n(F)})$. Then $f \in
dom (S_{T_n(F)})$ which implies that $U^{\ast}_{T_n(F)}f \in dom
(U_{T_n(F)})$. Now if $ran (U^{\ast}_{T_n(F)}) \subseteq dom
(U_{T_n(F)})$, it is clear that $dom (U^{\ast}_{T_n(F)}) \subseteq
dom (S_{T_n(F)})$ and the other inclusion is given in (i).
Therefore,\\ $dom (S_{T_n(F)})=dom (U^{\ast}_{T_n(F)})$.
\end{proof}


\bigskip

For a given operator $T\in B(\mathcal{H})$, let
\begin{eqnarray*}
\mathcal{V}(T):=\left\{ \phi \in \mathcal{H}: \{T^k\phi \}_{k\in
\mathbb{N}_0}\,\ \mbox{is a frame for}\,\ \mathcal{H} \right\}.
\end{eqnarray*}
\smallskip Also, let
\begin{eqnarray*}
\mathcal{E}(\mathcal{H}):=\left\{T\in B(\mathcal{H}):
\{T^k\phi\}_{k\in \mathbb{N}_0}\,\ \mbox{is a frame for}\,\
\mathcal{H},\,\ \mbox{for some}\,\ \phi \in \mathcal{H} \right\}.
\end{eqnarray*}
\smallskip

Let $T \in B(\mathcal{H})$ be an operator for which there exists
some $f \in \mathcal{H}$ such that $\{T^kf\}_{k=0}^{\infty}$ is a
frame for $\mathcal{H}$. A natural question to ask is whether
there exist other vectors $\phi \in \mathcal{H}$ for which
$\{T^k\phi\}_{k=0}^{\infty}$ is also a frame for $\mathcal{H}$. In
the following theorem we consider a viewpoint of this query; we
discussed the size of the set of vectors $\phi \in \mathcal{H}$
for which $\{T^n\phi\}_{k=0}^{\infty}$ is a frame for
$\mathcal{H}$.

\bigskip

\begin{prop}
\label{thm.1} Let $T \in B(\mathcal{H})$. Assume that $T$ is
invertible. Then
\begin{eqnarray*}
\mathcal{V}(T)=\bigcap_{f \in \mathcal{V}(T), k\in
\mathbb{N}}\mathbf{B}(f,k),
\end{eqnarray*}
where,
\begin{eqnarray*}
\mathbf{B}(f,k)=\bigcup_{n \in \mathbb{N}_0} \{\phi \in
\mathcal{H}: \|T^n\phi-f\|<\frac{1}{k}\}.
\end{eqnarray*}

\end{prop}

\begin{proof}
For any $f \in \mathcal{V}(T)$, it is obvious that $T \in
\mathcal{E}(\mathcal{H})$. Now let $B_{f,k}$'s are open balls
centered at $f\in \mathcal{V}(T)$ and with radius $\frac{1}{k}, (k
\in \mathbb{N})$. Then by continuity, for any $n \in \mathbb{N}_0$
$$(T^n)^{-1}B_{f,k}=\{\phi \in \mathcal{H}:
\|T^n\phi-f\|<\frac{1}{k}\},$$ is open in $\mathcal{H}$.
Therefore, for any $k \in \mathbb{N}$
\begin{eqnarray}\label{eq.200}
\mathbf{B}(f,k):=\bigcup_{n \in \mathbb{N}_0}(T^n)^{-1}B_{f,k},
\end{eqnarray}
 is open in $\mathcal{H}$.\\

\textbf{Claim}:
$$\mathcal{V}(T)=\bigcap_{f\in \mathcal{V}(T),\,\
k\in \mathbb{N}}\mathbf{B}(f,k).$$

Let $\phi \in \mathcal{V}(T)$. Then for any $f \in \mathcal{V}(T)$
and any $k\in \mathbb{N}$, there exists $n \in \mathbb{N}_0$ such
that $T^n\phi \in B_{f,k}$, (because for example, if $n=0$ then
$T^n\phi=\phi \in \mathcal{V}(T)$ and also we know that
$B_{f,k}$'s are open balls centered at $f\in \mathcal{V}(T)$ and
with radius $\frac{1}{k}, (k \in \mathbb{N})$ and hence,
$T^n\phi=\phi \in B_{f,k}$), it follows that $\phi \in
(T^n)^{-1}B_{f,k}$. Therefore,
\begin{eqnarray*}
\phi \in \bigcap_{k \in \mathbb{N}}\bigcup_{n \in \mathbb{N}_0
}(T^{n})^{-1}B_{f,k}.
\end{eqnarray*}
 It follows that, $\phi \in \bigcap_{k \in
\mathbb{N}}\mathbf{B}(f,k)$, for any $f \in \mathcal{V}(T)$.
Conversely, let $\phi \in \bigcap_{f\in \mathcal{V}(T), k \in
\mathbb{N}}\mathbf{B}(f,k)$, then for any $f \in \mathcal{V}(T)$
and any $k \in \mathbb{N}$,
\begin{eqnarray}\label{eqn.299}
\phi \in \mathbf{B}(f,k).
\end{eqnarray}
Hence, for any $f \in \mathcal{V}(T)$ and any $k \in \mathbb{N}$,
from (\ref{eq.200}) and (\ref{eqn.299}) we conclude that there
exists $m \in \mathbb{N}_0$ such that $\phi \in
(T^m)^{-1}B_{f,k}$, it follows that
\begin{eqnarray}\label{eqn.300}
T^m\phi \in B_{f,k}.
\end{eqnarray}
We now need the following fact in
the sequel:\\

\textit{Fact.} For any $f\in \mathcal{V}(T)$ and for $k \in
\mathbb{N}$ sufficiently large, $B_{f,k} \subset \mathcal{V}(T)$.

{\textit{proof of the fact.}} We have to show that for any $f\in
\mathcal{V}(T)$, if $\|f-\phi\|$ is small enough then $\phi \in
\mathcal{V}(T)$, i.e., $\{\phi_i\}_{i\in \mathbb{N}_0}:=\{T^i\phi
\}_{i\in \mathbb{N}_0}$ is a frame for $\mathcal{H}$. For this
purpose, it is enough that $\{\phi_i\}_{i\in \mathbb{N}_0}$
satisfy the assumptions of Theorem (2) in \cite{Casazza.1997}. Let
$A$ be a lower frame bound for the frame $\{f_i\}_{i\in
\mathbb{N}_0}:=\{T^if \}_{i\in \mathbb{N}_0}$. We have
\begin{eqnarray*}
\|\sum_{i=0}^n c_i(f_i-\phi_i)\|&=&\sup_{\|g\|=1}|\langle
\sum_{i=0}^n c_i(f_i-\phi_i), \;g\rangle|\\
&\leq& \sup_{\|g\|=1}\sum_{i=0}^n|c_i\langle(f_i-\phi_i),\;g\rangle|\\
\mbox{(Cauchy-Schwarz inequality)}&\leq&
(\sum_{i=0}^n|c_i|^2)^{1/2}\sup_{\|g\|=1} (\sum_{i=0}^n |\langle
f_i-\phi_i,\;g\rangle|^2)^{1/2}.
\end{eqnarray*}
On the other hand,
\begin{eqnarray*}
\sup_{\|g\|=1} (\sum_{i=0}^n |\langle
f_i-\phi_i,\;g\rangle|^2)^{1/2}&=&\sup_{\|g\|=1} (\sum_{i=0}^n
|\langle T^i(f-\phi),\;g\rangle|^2)^{1/2}\\
&\leq& \sup_{\|g\|=1} (\sum_{i=0}^n
\|T\|^{2i}\|f-\phi\|^2\|g\|^2)^{1/2}\\
&\leq& \frac{1}{k}(\sum_{i=0}^n\|T\|^{2i})^{1/2}.
\end{eqnarray*}
Therefore, we have
\begin{eqnarray*}
\|\sum_{i=0}^n c_i(f_i-\phi_i)\|\leq
\mu(\sum_{i=0}^n|c_i|^2)^{1/2},
\end{eqnarray*}
where $\mu:=\frac{1}{k}(\sum_{i=0}^n\|T\|^{2i})^{1/2}$. Choose $k
\in \mathbb{N}$ sufficiently large such that $\mu<\sqrt{A}$.
Therefore, $\{\phi_i\}_{i\in \mathbb{N}_0}:=\{T^i\phi \}_{i\in
\mathbb{N}_0}$ satisfy in \cite[Theorem (2)]{Casazza.1997} with
$\lambda_1=\lambda_2=0$.

Hence by (\ref{eqn.300}) and the previous Fact we have $T^m\phi
\in \mathcal{V}(T)$, it follows that $\{T^n(T^m\phi)\}_{n \in
\mathbb{N}_0}$ is a frame for $\mathcal{H}$. By adding of a finite
elements of points $\{\phi, T\phi,\cdots T^{m-1}\phi\}$ to the
system $\{T^n(T^m\phi)\}_{n=0}^{\infty}$, yields that the system
$\{T^n \phi\}_{n=0}^{\infty}$ is a frame for $\mathcal{H}$, that
is, $\phi \in \mathcal{V}(T)$ and this completes the proof.
\end{proof}

\bigskip

It is well-known that for any $T\in B(\mathcal{H})$, $\sigma(T)$
is included in the ball $B(0, \|T\|)$. Therefore, $r(T)\leq \|T\|$
for any $T\in B(\mathcal{H})$. In \cite{Christensen.2018}, it is
proved that $\mathcal{E}(\mathcal{H})$ does not form an open set
in $B(\mathcal{H})$. The next proposition yields that the set of
invertible elements in $\mathcal{E}(\mathcal{H})$ is relatively
open in $\mathcal{E}(\mathcal{H})$.
\bigskip
\begin{prop}
The set of invertible operators in $\mathcal{E}(\mathcal{H})$ is
relatively open in $\mathcal{E}(\mathcal{H})$.
\end{prop}
\begin{proof}
Suppose that $\widetilde{\mathcal{E}}(\mathcal{H})$ is the set of
invertible operators in $\mathcal{\mathcal{E}}(\mathcal{H})$ and
pick $T \in \widetilde{\mathcal{E}}(\mathcal{H})$. For each $U \in
\mathcal{E}(\mathcal{H})$ we have

\begin{eqnarray} \label{eqn.22}
U=T(I-T^{-1}(T-U)).
\end{eqnarray}

Assume that $\mathcal{N}_r(T)$ is an open neighborhood centered at
$T$ and with radius $r:=\|T^{-1}\|^{-1}$, i.e., for every $U \in
\mathcal{N}_r(T)$ we have $$\|T-U\| < \|T^{-1}\|^{-1}.$$ Also, we
have
\begin{eqnarray*} \label{eqn.23}
r(T^{-1}(T-U))&\leq& \|T^{-1}(T-U)\|\\
&\leq& \|T^{-1}\|\|(T-U)\|\\
&<& \|T^{-1}\|\|T^{-1}\|^{-1}\\
 &=&1.
\end{eqnarray*}
Hence the value $1$ is not in $\sigma(T^{-1}(T-U))$, that is, the
operator $I-T^{-1}(T-U)$ is invertible and therefore, by
(\ref{eqn.22}) $U$ is invertible and belongs to
$\widetilde{E}(\mathcal{H})$. We have proved that there exists an
open ball around $T$ made of bounded invertible operators in
$\mathcal{E}(\mathcal{H})$, that is,
$\widetilde{\mathcal{E}}(\mathcal{H})$ is open.
\end{proof}


\begin{rem} It is clear that $\mathcal{E}(\mathcal{H})$ cannot be dense
in $B(\mathcal{H})$ with respect to the norm topology. Indeed, by
Prop. 2.2. in \cite{Christensen.2017} every operator $T \in
\mathcal{E}(\mathcal{H})$ has norm greater or equal than $1$. So
the norm topology is not always the most natural topology on
$B(\mathcal{H})$. It is often more useful to consider the weakest
topology on $B(\mathcal{H})$, so-called the strong operator
topology, it is defined by the family of seminorms $\{p_h: h \in
\mathcal{H}\}$, where $p_h(T)=\|Th\|$.
\end{rem}

We conclude this note by raising the following questions:

\bigskip

\textbf{Q1}. For $T \in B(\mathcal{H})$, pick $c=\|T\|+\alpha,\,\
(0<\alpha<1)$ and replace the norm topology by the strong operator
topology. What can we say for the size of the set of all operators
in $\mathcal{E}(\mathcal{H})$ with the norm of at most $c$?\\

\textbf{Q2}. Do there exist $T \in \mathcal{E}(\mathcal{H})$ such
that
$T^{-1}\in \mathcal{E}(\mathcal{H})$?\\

\textbf{Q3}. Let $T \in B(\mathcal{H})$ is invertible and that
there exist a dense subset $D\subset \mathcal{H}$ such that $T^k$
and $(T^{-1})^k$ tends to zero as $k\rightarrow \infty$ on $D$. Do
there exist $\varphi \in \mathcal{H}$ such that
$\{T^k\phi\}_{k=0}^{\infty}$ and
$\{(T^{-1})^k\phi\}_{k=0}^{\infty}$ are frames in $\mathcal{H}$?

\bigskip


\bigskip








\bibliographystyle{plain}
\end{document}